\documentclass[a4paper,12pt]{amsart}
\usepackage{etex}
\usepackage{fullpage}
\usepackage{txfonts}
\usepackage{latexsym}
\usepackage{amssymb}
\usepackage{stmaryrd}
\usepackage{mathrsfs}
\usepackage[sans]{dsfont}
\usepackage[all]{xy}
\usepackage{proof}
\usepackage{amsaddr}
\def\limp{\Rightarrow}
\def\liff{\Leftrightarrow}
\def\dand{\mathbin{\dot{{\land}}}}
\def\dor{\mathbin{\dot{{\lor}}}}
\def\dimp{\mathbin{\dot{{\to}}}}
\def\dbot{\dot{\bot}}
\def\dtop{\dot{\top}}
\def\dbigvee{\mathbin{\dot{{\bigvee}}}}
\def\dbigwedge{\mathbin{\dot{{\bigwedge}}}}
\def\<{\langle}
\def\>{\rangle}

\def\Pow{\mathfrak{P}}

\def\op{\mathrm{op}}
\def\id{\mathrm{id}}

\def\<{\langle}
\def\>{\rangle}
\def\P{\mathsf{P}}

\def\Set{\mathbf{Set}}
\def\Pos{\mathbf{Pos}}
\def\HA{\mathbf{HA}}
\def\BA{\mathbf{BA}}

\let\phi=\varphi
\def\tphi{\tilde{\phi}}
\def\pullbackbox#1{\hbox to#1{\vbox to#1{\hbox to#1{\hfil}\vfil\hrule}\hskip-0.4pt\vrule}}
\def\pullback#1{\hbox to 0pt{\vbox to 0pt{\pullbackbox{#1}\vss}\hss}}

\def\tr{\textit{tr}}
\def\sem#1{\llbracket#1\rrbracket}

\def\bigsem#1{\bigl\llbracket#1\bigr\rrbracket}
\def\Bigsem#1{\Bigl\llbracket#1\Bigr\rrbracket}
\def\csem#1{\langle\mskip-4mu\langle#1\rangle\mskip-4mu\rangle}

\def\bigcsem#1{\bigl\langle\mskip-4mu\bigl\langle#1%
\bigr\rangle\mskip-4mu\bigr\rangle}
\def\Bigcsem#1{\Bigl\langle\mskip-4mu\Bigl\langle#1%
\Bigr\rangle\mskip-4mu\Bigr\rangle}

\def\ent{\vdash}
\def\tnent{\mathrel{{\dashv}{\vdash}}}
\def\A{\mathscr{A}}
\def\B{\mathscr{B}}
\def\cle{\preccurlyeq}

\def\bigmeet{\bigcurlywedge}

\def\Powup{\Pow_{\uparrow}}
\def\Powfin{\Pow_{\text{fin}}}
\def\ds{\displaystyle}
\outer\long\def\COUIC#1{}
\newtheorem{theorem}{Theorem}[section]
\newtheorem{proposition}[theorem]{Proposition}
\newtheorem{fact}[theorem]{Fact}
\newtheorem{lemma}[theorem]{Lemma}
\newtheorem{corollary}[theorem]{Corollary}
\theoremstyle{definition}
\newtheorem{definition}[theorem]{Definition}
\newtheorem{notation}[theorem]{Notation}
\newtheorem{remark}[theorem]{Remark}
\newtheorem{remarks}[theorem]{Remarks}

\begin{document}
\title{Implicative algebras II:\\
  completeness w.r.t. Set-based triposes}
\author{Alexandre Miquel}
\address{Instituto de Matem\'atica y Estad\'istica
  Prof. Ing. Rafael Laguardia\\
  Facultad de Ingenier\'ia -- Universidad de la Rep\'ublica\\
  Julio Herrera y Reissig 565 -- Montevideo C.P. 11300 -- URUGUAY}
\date{November, 2020}

\begin{abstract}
  We prove that all $\Set$-based triposes are implicative triposes.
\end{abstract}
\maketitle

\section{Introduction}

In~\cite{Miq20}, we introduced the notion of \emph{implicative
  algebra}, a simple algebraic structure that is intended to factorize
the model constructions underlying forcing and realizability, both in
intuitionistic and classical logic.
We showed that this algebraic structure induces a large class of
($\Set$-based) triposes---the \emph{implicative triposes}---, that
encompasses all (intuitionistic and classical) forcing triposes, all
classical realizability triposes (both in the sense of
Streicher~\cite{Str13} and Krivine~\cite{Miq20}) as well as all the
intuitionistic realizability triposes induced by partial combinatory
algebras (in the style of Hyland, Johnstone and Pitts~\cite{HJP80}).

The aim of this paper is to prove that the class of implicative
triposes actually encompasses all $\Set$-based triposes, in the sense
that:
\begin{theorem}\label{th:Thm}
  Each $\Set$-based tripos is (isomorphic to) an implicative tripos.
\end{theorem}

For that, we first recall some notions about triposes and
implicative algebras.

\subsection{$\Set$-based triposes}
\label{ss:SetBasedTriposes}
In what follows, we write:
\begin{itemize}
\item $\Set$ the category of sets equipped with all functions;
\item $\Pos$ the category of posets equipped with monotonic
  functions;
\item $\HA$ the category of Heyting algebras equipped with
  the corresponding morphisms.
\end{itemize}
In the category~$\Set$, we write:
\begin{itemize}
\item $1$ the terminal object (i.e. a fixed singleton);
\item $1_X:X\to 1$ the unique map from a given set~$X$ to~$1$;
\item $X\times Y$ the Cartesian product of two sets~$X$ and $Y$,
  and
\item $\pi_{X,Y}:X\times Y\to X$ and $\pi'_{X,Y}:X\times Y\to Y$
  the associated projections.
\item Finally, given two functions $f:Z\to X$ and $g:Z\to Y$,
  we write $\<f,g\>:Z\to X\times Y$ the unique function such that
  $\pi_{X,Y}\circ\<f,g\>=f$ and $\pi'_{X,Y}\circ\<f,g\>=g$.
\end{itemize}

\begin{definition}[$\Set$-based triposes]\label{d:Tripos}
  A \emph{$\Set$-based tripos} is a functor $\P:\Set^{\op}\to\HA$ that
  fulfills the following three conditions:
  \begin{enumerate}
  \item For each map $f:X\to Y$ ($X,Y\in\Set$), the corresponding map
    $\P{f}:\P{Y}\to\P{X}$ has left and right adjoints in~$\Pos$, that
    are monotonic maps $\exists{f},\forall{f}:\P{X}\to\P{Y}$ such that
    $$\begin{array}{rcl}
      \exists{f}(p)\le q&\liff&p\le\P{f}(q) \\[3pt]
      q\le\forall{f}(p)&\liff&\P{f}(q)\le p \\
    \end{array}\eqno(\text{for all}~p\in\P{X},~q\in\P{Y})$$
  \item Beck-Chevalley condition. Each pullback square in $\Set$ (on
    the left-hand side) induces the following two commutative diagrams
    in~$\Pos$ (on the right-hand side):
    $$\begin{array}{@{}c@{}}
      \xymatrix{
        X\pullback{6pt}\ar[r]^{f_1}\ar[d]_{f_2}& X_1\ar[d]^{g_1} \\
        X_2\ar[r]_{g_2} & Y \\
      }\\
    \end{array}\qquad{\limp}\qquad
    \begin{array}{c@{\qquad}c}
      \xymatrix{
        \P{X}\ar[r]^{\exists{f_1}}& \P{X_1} \\
        \P{X_2}\ar[u]^{\P{f_2}}\ar[r]_{\exists{g_2}} &
        \P{Y}\ar[u]_{\P{g_1}} \\
      } &
      \xymatrix{
        \P{X}\ar[r]^{\forall{f_1}}& \P{X_1} \\
        \P{X_2}\ar[u]^{\P{f_2}}\ar[r]_{\forall{g_2}} &
        \P{Y}\ar[u]_{\P{g_1}} \\
      }\\
    \end{array}$$
    That is:\quad
    $\exists{f_1}\circ\P{f_2}~=~\P{g_1}\circ\exists{g_2}$\quad
    and\quad $\forall{f_1}\circ\P{f_2}~=~\P{g_1}\circ\forall{g_2}$.
  \item The functor $\P:\Set^{\op}\to\HA$ has a \emph{generic
    predicate}, that is: a predicate $\tr_{\Sigma}\in\P\Sigma$ (for
    some set~$\Sigma$) such that for all sets~$X$, the following map
    is surjective:
    $$\begin{array}{r@{~{}~}c@{~{}~}l}
      \Sigma^X&\to&\P{X} \\
      \sigma&\mapsto&\P\sigma(\tr_{\Sigma}) \\
    \end{array}$$
  \end{enumerate}
\end{definition}

\begin{remarks}
  (1)~Given a map $f:X\to Y$, the left and right adjoints
  $\exists{f},\forall{f}:\P{X}\to\P{Y}$ of the substitution map
  $\P{f}:\P{Y}\to\P{X}$ are always monotonic functions (due to the
  adjunction), but in general they are not morphisms of Heyting
  algebras.
  Moreover, both correspondences $f\mapsto\exists{f}$ and
  $f\mapsto\forall{f}$ are functorial, in the sense that
  $$\begin{array}{c@{\qquad\qquad}c}
    \exists(g\circ f)~=~\exists{g}\circ\exists{f}&
    \exists\id_X~=~\id_{\P{X}}\\
    \forall(g\circ f)~=~\forall{g}\circ\forall{f}&
    \forall\id_X~=~\id_{\P{X}}\\
  \end{array}$$
  for all sets~$X$, $Y$, $Z$ and for all maps $f:X\to Y$ and
  $g:Y\to Z$.
  So that we can see~$\exists$ and $\forall$ as covariant functors
  from~$\Set$ to~$\Pos$, whose action on sets is given by
  $\exists{X}=\forall{X}=\P{X}$.\par
  (2)~When defining the notion of tripos, some authors~\cite{Pit02}
  require that the Beck-Chevalley condition hold only for the pullback
  squares of the form 
  $$\begin{array}{@{}c@{}}
    \xymatrix@C=36pt{
      Z\times X\pullback{6pt}
      \ar[r]^{\pi_{Z,X}}\ar[d]_{f\times\id_X}&Z\ar[d]^{f} \\
      Z'\times X\ar[r]_{\pi_{Z',X}}&Z'\\
    }
  \end{array}\eqno(\text{in}~\Set)$$
  Here, we follow~\cite{HJP80,Miq20} by requiring that the
  Beck-Chevalley condition hold more generally for all pullback
  squares in~$\Set$.\par
  (3)~The meaning of the generic predicate $\tr_{\Sigma}\in\P\Sigma$
  will be explained in Section~\ref{ss:GenPred}.
\end{remarks}

Let us also recall that:
\begin{definition}[Isomorphism of triposes]
  Two triposes $\P,\P':\Set^{\op}\to\HA$ are \emph{isomorphic} when
  there is a natural isomorphism $\phi:\P\limp\P'$.
\end{definition}

\begin{remark}
  By a natural isomorphism $\phi:\P\limp\P'$, we mean any family of
  isomorphisms $\phi_X:\P{X}\to\P'{X}$ (indexed by $X\in\Set$) such
  that for all maps $f:X\to Y$ ($X,Y\in\Set$), the following diagram
  commutes:
  $$\xymatrix{
    X\ar[d]_f & \P{X}\ar[r]^{\phi_X}_{\sim} & \P'{X} \\
    Y & \P{Y}\ar[u]^{\P{f}}\ar[r]_{\phi_Y}^{\sim} &
    \P'{Y}\ar[u]_{\P'{f}} \\
  }$$
  Note that here, the notion of isomorphism can be taken indifferently
  in the sense of $\HA$ or $\Pos$, since a map $\phi_X:\P{X}\to\P'{X}$
  is an isomorphism of Heyting algebras if and only if it is an
  isomorphism between the underlying posets.
\end{remark}

To conclude this presentation of triposes, we recall a few facts about
left and right adjoints:
\begin{lemma}
  For all maps $f:X\to Y$ and for all predicates $p,p'\in\P{X}$, we
  have:
  $$\begin{array}{l@{\qquad\qquad}l}
    \exists{f}(p\lor p')~=~\exists{f}(p)\lor\exists{f}(p')&
    \exists{f}(\bot_X)~=~\bot_Y \\[3pt]
    \forall{f}(p\land p')~=~\forall{f}(p)\land\forall{f}(p')&
    \forall{f}(\top_X)~=~\top_Y \\
  \end{array}$$
\end{lemma}

\begin{proof}
  Let us treat the case of $\exists{f}$.
  For all $q\in\P{Y}$, we have
  $$\begin{array}{r@{\quad}c@{\quad}l}
    \exists{f}(p\lor p')\le q
    &\text{iff}& p\lor p'\le\P{f}(q) \\
    &\text{iff}& p\le\P{f}(q)~{}~\text{and}~{}~p'\le\P{f}(q)\\
    &\text{iff}& \exists{f}(p)\le q~{}~\text{and}~{}~
    \exists{f}(p')\le q\\ 
    &\text{iff}& \exists{f}(p)\lor\exists{f}(p')\le q\\
  \end{array}$$
  hence $\exists{f}(p\lor p')=\exists{f}(p)\lor\exists{f}(p')$.
  Moreover, we have $\bot_X\le\P{f}(\bot_Y)$, hence
  $\exists{f}(\bot_X)\le\bot_Y$, and thus
  $\exists{f}(\bot_X)=\bot_Y$.
  The proofs of the corresponding properties for $\forall{f}$
  proceed dually.
\end{proof}

\begin{lemma}\label{l:SimplAdj}
  Given a map $f:X\to Y$:
  \begin{enumerate}
  \item If $f$ has an inverse (i.e.\ $f$ is bijective), then
    $\exists{f}$ and $\forall{f}$ are the inverse of $\P{f}$:
    $$\exists{f}~=~\forall{f}~=~\P{f^{-1}}~=~(\P{f})^{-1}\,.$$
  \item If $f$ has a right inverse, then $\exists{f}$ and~$\forall{f}$
    are left inverses of $\P{f}$:
    $$\exists{f}\circ\P{f}~=~\forall{f}\circ\P{f}~=~\id_{\P{Y}}\,.$$
  \item If $f$ has a left inverse, then $\exists{f}$ and~$\forall{f}$
    are right inverses of $\P{f}$:
    $$\P{f}\circ\exists{f}~=~\P{f}\circ\forall{f}~=~\id_{\P{X}}\,.$$
  \end{enumerate}
\end{lemma}

\begin{proof}
  (1)~If $f:X\to Y$ is bijective, then for all $p\in\P{X}$ and
  $q\in\P{Y}$, we have
  \begin{itemize}
  \item $\exists{f}(p)\le q
    ~{}~\text{iff}~{}~p\le\P{f}(q)
    ~{}~\text{iff}~{}~(\P{f})^{-1}(p)\le q$,\quad
    hence:\quad $\exists{f}=(\P{f})^{-1}(p)=\P{f^{-1}}(p).$
  \item $q\le\forall{f}(p)
    ~{}~\text{iff}~{}~\P{f}(q)\le p
    ~{}~\text{iff}~{}~q\le(\P{f})^{-1}(p)$,\quad
    hence:\quad $\forall{f}=(\P{f})^{-1}(p)=\P{f^{-1}}(p).$
  \end{itemize}
  (2)~Let $g:Y\to X$ such that $f\circ g=\id_Y$.
  By functoriality, we get $\P{g}\circ\P{f}=\id_{\P{Y}}$, hence the
  map $\P{f}:\P{Y}\to\P{X}$ is an embedding of posets.
  For all $q,q'\in\P{Y}$, we thus have:
  \begin{itemize}
  \item $\exists{f}(\P{f}(q))\le q'
    ~{}~\text{iff}~{}~\P{f}(q)\le\P{f}(q')
    ~{}~\text{iff}~{}~q\le q'$,\quad
    hence\quad $\exists{f}\circ\P{f}=\id_{\P{Y}}$.
  \item $q'\le\forall{f}(\P{f}(q))
    ~{}~\text{iff}~{}~\P{f}(q')\le\P{f}(q)
    ~{}~\text{iff}~{}~q'\le q$,\quad
    hence\quad $\forall{f}\circ\P{f}=\id_{\P{Y}}$.
  \end{itemize}
  (3)~Let $g:Y\to X$ such that $g\circ f=\id_X$.
  By functoriality, we get
  $\forall{g}\circ\forall{f}=\exists{g}\circ\exists{f}=\id_{\P{X}}$,
  hence $\exists{f},\forall{f}:\P{X}\to\P{Y}$ are embeddings of
  posets.
  For all $p,p'\in\P{X}$, we thus have:
  \begin{itemize}
  \item $p'\le\P{f}(\exists{f}(p))
    ~{}~\text{iff}~{}~\exists{f}(p')\le\exists{f}(p)
    ~{}~\text{iff}~{}~p'\le p$,\quad
    hence\quad $\P{f}\circ\exists{f}=\id_{\P{X}}$.
  \item $\P{f}(\forall{f}(p))\le p'
    ~{}~\text{iff}~{}~\forall{f}(p)\le\forall{f}(p')
    ~{}~\text{iff}~{}~p\le p'$,\quad
    hence\quad $\P{f}\circ\forall{f}=\id_{\P{X}}$.\qedhere
  \end{itemize}
\end{proof}

\subsection{Implicative algebras}\label{ss:ImpAlg}
Recall that:
\begin{itemize}
\item An \emph{implicative structure} is a complete (meet-semi)lattice
  $(\A,{\cle})$ equipped with a binary operation $({\to}):\A^2\to\A$
  such that:
  \begin{enumerate}
  \item[(1)] If $a'\cle a$ and $b\cle b'$,\ \ then\ \
    $(a\to b)\cle(a'\to b')$.
  \item[(2)] For all $a\in\A$ and $B\subseteq\A$, we have:\quad
    $\ds a\to\bigmeet_{b\in B}\!\!b~=~\bigmeet_{b\in B}\!(a\to b)$.
  \end{enumerate}
\item A \emph{separator} of an implicative structure
  $(\A,{\cle},{\to})$ is a subset $S\subseteq\A$ such that:
  \begin{enumerate}
  \item[(1)] If $a\in S$ and $a\cle a'$, then $a'\in S$.
  \item[(2)] $\bigmeet_{a,b\in\A}(a\to b\to c)
    ~{}~({=}~\mathbf{K}^{\A})~\in~S$\quad and\\
    $\bigmeet_{a,b,c\in\A}((a\to b\to c)\to(a\to b)\to a\to c)
    ~{}~({=}~\mathbf{S}^{\A})~\in~S$.
  \item[(3)] If $(a\to b)\in S$ and $a\in S$, then $b\in S$.
  \end{enumerate}
  Moreover, we say that a separator $S\subseteq\A$ is
  \emph{classical} when
  \begin{enumerate}
  \item[(4)] $\bigmeet_{a,b\in\A}(((a\to b)\to a)\to a)
    ~{}~({=}~\mathbf{cc}^{\A})~\in~S$.
  \end{enumerate}
  \medbreak
\item An \emph{implicative algebra} is a quadruple
  $(\A,{\cle},{\to},S)$ where $(\A,{\cle},{\to})$ is an implicative
  structure and where $S\subseteq\A$ is a separator.
  An implicative algebra is \emph{classical} when the underlying
  separator $S\subseteq\A$ is classical.
\end{itemize}

\bigbreak
In~\cite[\S4]{Miq20}, we have seen that each implicative algebra
$(\A,{\cle},{\to},S)$ induces a $\Set$-based tripos
$\P:\Set^{\op}\to\HA$ that is defined as follows:
\begin{itemize}
\item For each set $X$, the corresponding Heyting algebra $\P{X}$ is
  defined as the poset reflection of the preordered set
  $(\A^X,\ent_{S[X]})$ whose preorder $\ent_{S[X]}$ is given by
  $$a\ent_{S[X]}b\quad\liff\quad
  \bigmeet_{x\in X}(a_x\to b_x)~\in~S
  \eqno(\text{for all}~a,b\in\A^X)$$
  The quotient $\A^X/{\tnent_{S[X]}}~({=}~\P{X})$ is also written
  $\A^X/S[X]$.
\item For each map $f:X\to Y$, the corresponding substitution map
  $\P{f}:\P{Y}\to\P{X}$ is the (unique) morphism of Heyting algebras
  that factors the map $\A^f:\A^Y\to\A^X$ through the quotients
  $\P{Y}=\A^Y/S[Y]$ and $\P{X}=\A^X/S[X]$.
\end{itemize}
The aim of this paper is thus to show that all $\Set$-based triposes
are of this form.

\section{Anatomy of a $\Set$-based tripos}

The results presented in this section are essentially taken
from~\cite{HJP80,Pit81}.

\subsection{The generic predicate}
\label{ss:GenPred}
From now on, we work with a fixed tripos $\P:\Set^{\op}\to\HA$.

From the definition, $\P$ has a \emph{generic predicate}, that is: a
predicate $\tr_{\Sigma}\in\P\Sigma$ (for some set $\Sigma$) such that for
each set~$X$, the corresponding `decoding map' is surjective:
$$\begin{array}{r@{~{}~}c@{~{}~}l}
  \sem{\_}_X~:~\Sigma^X&\to&\P{X} \\
  \sigma&\mapsto&\P\sigma(\tr_{\Sigma}) \\
\end{array}$$

Intuitively, $\Sigma$ can be thought of as the set of (codes of)
propositions, whereas $\Sigma^X$ can be thought of as the set of
\emph{propositional functions over~$X$}. 
The above condition thus expresses that each predicate $p\in\P{X}$ is
represented by at least one propositional function
$\sigma\in\Sigma^X$ such that $\sem{\sigma}_X=p$, which we shall
call a \emph{code} for the predicate~$p$.

\begin{notation}
  Given a code $\sigma\in\Sigma^X$, the predicate
  $\sem{\sigma}_X\in\P{X}$ will be often written
  $\sem{\sigma_x}_{x\in{X}}$.
  In particular, given an individual code $\xi\in\Sigma$, we write
  $(\xi)_{\_\in1}\in\Sigma^1$ the corresponding constant family
  indexed by the singleton~$1$, and $\sem{\xi}_{\_\in1}\in\P{1}$ the
  associated predicate.
\end{notation}

\begin{fact}[Naturality of $\sem{\_}_X$]
  The decoding map $\sem{\_}_X:\Sigma^X\to\P{X}$ is natural in~$X$, in
  the sense that for each map $f:X\to Y$, the following diagram
  commutes:
  $$\begin{array}{@{}c@{}}
    \xymatrix{
      \Sigma^X\ar[r]^{\sem{\_}_X}&\P{X}\\
      \Sigma^Y\ar[r]_{\sem{\_}_Y}\ar[u]^{\_\circ f}&
      \P{Y}\ar[u]_{\P{f}} \\
    }\\
  \end{array}\qquad
  \sem{\sigma\circ f}_X~=~\P{f}(\sem{\sigma}_Y)
  \eqno(\sigma\in\Sigma^Y)$$
\end{fact}

\begin{proof}
  For all $\sigma\in\Sigma^Y$, we have\ \
  $\sem{\sigma\circ f}_X=\P(\sigma\circ f)(\tr_{\Sigma})=
  \P{f}(\P\sigma(\tr_{\Sigma}))=\P{f}(\sem{\sigma}_Y)$.
\end{proof}

\begin{remarks}[Non-uniqueness of generic predicates]
  It is important to observe that in a $\Set$-based tripos~$\P$, the
  generic predicate is never unique.
  \begin{itemize}
  \item Indeed, given a generic predicate $\tr_{\Sigma}\in\P\Sigma$
    and a surjection $h:\Sigma'\to\Sigma$, we can always construct
    another generic predicate $\tr_{\Sigma'}\in\P\Sigma'$ by letting
    $\tr_{\Sigma'}=\P{h}(\tr_{\Sigma})$%
    \footnote{To prove that $\tr_{\Sigma'}\in\P\Sigma'$ is another
      generic predicate, we actually need a right inverse of
      $h:\Sigma'\to\Sigma$, which exists by~(AC).
      Without (AC), the same argument works by replacing `surjective'
      with `having a right inverse'.}.
    \COUIC{
      \begin{proof}
        Since $h:\Sigma'\to\Sigma$ is surjective, it has a right
        inverse by (AC), that is a map $h':\Sigma\to\Sigma'$ such that
        $h\circ h'=\id_{\Sigma}$. 
        Given a set $X$ and a predicate $p\in\P{X}$, we know that
        there is a code $\sigma\in\Sigma^X$ such that
        $\P\sigma(\tr_{\Sigma})=p$.
        Now letting $\sigma':=h'\circ\sigma\in\Sigma'^X$,
        we observe that
        $\P{\sigma'}(\tr_{\Sigma'})=
        \P(h'\circ\sigma)(\P{h}(\tr_{\Sigma}))=
        \P(h\circ h'\circ\sigma)(\tr_{\Sigma})=
        \P{\sigma}(\tr_{\Sigma})=p$.
      \end{proof}
    }
  \item More generally, if $\tr_{\Sigma}\in\P\Sigma$ and
    $\tr_{\Sigma'}\in\P\Sigma'$ are two generic predicates of the same
    tripos~$\P$, then there always exist two conversion maps
    $h:\Sigma'\to\Sigma$ and $h':\Sigma\to\Sigma'$ such that
    $\tr_{\Sigma'}=\P{h}(\tr_{\Sigma})$ and
    $\tr_{\Sigma}=\P{h'}(\tr_{\Sigma'})$.
  \end{itemize}
  In what follows, we will work with a fixed generic predicate
  $\tr_{\Sigma}\in\P\Sigma$.
\end{remarks}

\subsection{Defining connectives on~$\Sigma$}
\label{ss:DefConnSigma}
We want to show that the operations $\land$, $\lor$ and $\to$ on each
Heyting algebra $\P{X}$ ($X\in\Set$) can be derived from analogous
operations on the generic set $\Sigma$ of propositions.
For that, we pick codes
$({\dand}),({\dor}),({\dimp})\in\Sigma^{\Sigma\times\Sigma}$ such that
$$\begin{array}{r@{~{}~}c@{~{}~}l}
  \sem{\dand}_{\Sigma\times\Sigma}&=&
  \sem{\pi}_{\Sigma\times\Sigma}\land
  \sem{\pi'}_{\Sigma\times\Sigma}\\[3pt]
  \sem{\dor}_{\Sigma\times\Sigma}&=&
  \sem{\pi}_{\Sigma\times\Sigma}\lor
  \sem{\pi'}_{\Sigma\times\Sigma}\\[3pt]
  \sem{\dimp}_{\Sigma\times\Sigma}&=&
  \sem{\pi}_{\Sigma\times\Sigma}\to
  \sem{\pi'}_{\Sigma\times\Sigma}\\
\end{array}\eqno\begin{array}{r@{}}
  ({\in}~\P(\Sigma\times\Sigma))\\[3pt]
  ({\in}~\P(\Sigma\times\Sigma))\\[3pt]
  ({\in}~\P(\Sigma\times\Sigma))\\
  \end{array}$$
writing $\pi,\pi':\Sigma\times\Sigma\to\Sigma$ the two projections
from $\Sigma\times\Sigma$ to $\Sigma$.

\begin{proposition}\label{p:ConnSigma}
  Let $X$ be a set.
  For all codes $\sigma,\tau\in\Sigma^X$, we have
  $$\begin{array}{r@{~{}~}c@{~{}~}l}
    \sem{\sigma_x\dand\tau_x}_{x\in X}
    &=&\sem{\sigma}_X\land\sem{\tau}_X\\[3pt]
    \sem{\sigma_x\dor\tau_x}_{x\in X}
    &=&\sem{\sigma}_X\lor\sem{\tau}_X\\[3pt]
    \sem{\sigma_x\dimp\tau_x}_{x\in X}
    &=&\sem{\sigma}_X\to\sem{\tau}_X\\
  \end{array}\eqno\begin{array}{r@{}}
  ({\in}~\P{X})\\[3pt]
  ({\in}~\P{X})\\[3pt]
  ({\in}~\P{X})\\
  \end{array}$$
\end{proposition}

\begin{proof}
  Let us treat (for instance) the case of implication:
  $$\begin{array}[b]{r@{~{}~}c@{~{}~}l}
    \sem{\sigma_x\dimp\tau_x}_{x\in X}
    &=&\sem{({\dimp})\circ\<\sigma,\tau\>}_X
    ~=~\P\<\sigma,\tau\>\sem{{\dimp}}_{\Sigma\times\Sigma} \\
    &=&\P\<\sigma,\tau\>\bigl(\sem{\pi}_{\Sigma\times\Sigma}\to
    \sem{\pi'}_{\Sigma\times\Sigma}\bigr)\\
    &=&\P\<\sigma,\tau\>\bigl(\sem{\pi}_{\Sigma\times\Sigma}\bigr)\to
    \P\<\sigma,\tau\>\bigl(\sem{\pi'}_{\Sigma\times\Sigma}\bigr)\\
    &=&\sem{\pi\circ\<\sigma,\tau\>}_X\to
    \sem{\pi'\circ\<\sigma,\tau\>}_X
    ~=~\sem{\sigma}_X\to\sem{\tau}_X\,.\\
  \end{array}\eqno\mbox{\qedhere}$$
\end{proof}

Similarly, we pick codes $\dbot,\dtop\in\Sigma$ such that
$$\sem{\dbot}_{\_\in 1}=\bot_1\in\P(1)\qquad\text{and}\qquad
\sem{\dtop}_{\_\in 1}=\top_1\in\P(1)\,.$$
Again, we easily check that:
\begin{proposition}\label{p:UnitSigma}
  For each set $X$, we have:
  $$\begin{array}{r@{~{}~}r@{~{}~}c@{~{}~}l}
    \sem{\dbot}_{x\in X}&=&\bot_X\\[3pt]
    \sem{\dtop}_{x\in X}&=&\top_X\\
  \end{array}\eqno\begin{array}{r@{}}
  ({\in}~\P{X})\\[3pt]
  ({\in}~\P{X})\\
  \end{array}$$
\end{proposition}

\begin{proof}
  Let us treat (for instance) the case of $\top$:
  $$\sem{\dtop}_{x\in X}
  ~=~\sem{(\dtop)_{\_\in 1}\circ 1_X}_X
  ~=~\P{1_X}(\sem{\dtop}_{\_\in 1})
  ~=~\P{1_X}(\top_1)~=~\top_X$$
  (writing $1_X:X\to 1$ the unique map from $X$ to $1$).
\end{proof}

\subsection{Defining quantifiers on~$\Sigma$}
\label{ss:DefQuantSigma}

In this section, we propose to show that the adjoints
$$\exists{f},\forall{f}~:~\P{X}\to\P{Y}
\eqno(\text{for each}~f:X\to Y)$$
can be derived from suitable \emph{quantifiers} on the generic
set~$\Sigma$ of propositions.
For that, we consider the membership relation
$E:=\{(\xi,s):\xi\in s\}\subseteq\Sigma\times\Pow(\Sigma)$ and
write $e_1:E\to\Sigma$ and $e_2:E\to\Pow(\Sigma)$ the corresponding
projections.
We pick two codes
$\dbigvee,\dbigwedge\in\Sigma^{\Pow(\Sigma)}$ such that
$$\begin{array}{r@{~~}c@{~~}l}
  \sem{\dbigvee}_{\Pow(\Sigma)}&=&\exists{e_2}(\sem{e_1}_E)\\[3pt]
  \sem{\dbigwedge}_{\Pow(\Sigma)}&=&\forall{e_2}(\sem{e_1}_E)\\
\end{array}\eqno\begin{array}{r@{}}
({\in}~\P(\Pow(\Sigma)))\\[3pt]
({\in}~\P(\Pow(\Sigma)))\\
\end{array}$$

\begin{proposition}\label{p:QuantSigma}
  Given a code $\sigma\in\Sigma^X$ and a map $f:X\to Y$, we have:
  $$\begin{array}{r@{~{}~}c@{~{}~}l}
    \bigsem{\dbigvee\{\sigma_x:x\in f^{-1}(y)\}}_{y\in Y}
    &=&\exists{f}(\sem{\sigma}_X)\\[6pt]
    \bigsem{\dbigwedge\{\sigma_x:x\in f^{-1}(y)\}}_{y\in Y}
    &=&\forall{f}(\sem{\sigma}_X)\\
  \end{array}\eqno\begin{array}{r@{}}
  ({\in}~\P{Y})\\[6pt]
  ({\in}~\P{Y})\\
  \end{array}$$
\end{proposition}

\begin{proof}
  Let us define the map $h:Y\to\Pow(\Sigma)$ by
  $h(y):=\{\sigma_x:x\in f^{-1}(y)\}$ for all $y\in Y$.
  From this definition and from the definitions of $\dbigvee$,
  $\dbigwedge$, we get
  $$\begin{array}{rcl}
    \sem{\dbigvee\{\sigma_x:x\in f^{-1}(y)\}}_{y\in Y}
    &=&\sem{{\dbigvee}\circ h}_Y
    ~=~\P{h}\bigl(\sem{\dbigvee}_{\Pow(\Sigma)}\bigr)
    ~=~\P{h}\bigl(\exists{e_2}(\sem{e_1}_E)\bigr) \\[3pt]
    \sem{\dbigwedge\{\sigma_x:x\in f^{-1}(y)\}}_{y\in Y}
    &=&\sem{{\dbigwedge}\circ h}_Y
    ~=~\P{h}\bigl(\sem{\dbigwedge}_{\Pow(\Sigma)}\bigr)
    ~=~\P{h}\bigl(\forall{e_2}(\sem{e_1}_E)\bigr) \\
  \end{array}$$
  Let us now consider the set $G\subseteq\Sigma\times Y$ defined by
  $G:=\{(\sigma_x,f(x)):x\in X\}$ as well as the two functions
  $g:G\to Y$ and $g':G\to E$ given by
  $$g(\xi,y)~:=~y\qquad\text{and}\qquad
  g'(\xi,y)~:=~(\xi,h(y))
  \eqno(\text{for all}~(\xi,y)\in G)$$
  We observe that the following diagram is a pullback in~$\Set$:
  $$\xymatrix{
    G\ar[r]^{g}\ar[d]_{g'}\pullback{6pt} & Y\ar[d]^{h} \\
    E\ar[r]_{e_2}&\Pow(\Sigma) \\
  }$$
  Hence $\P{h}\circ\exists{e_2}=\exists{g}\circ\P{g'}$ and
  $\P{h}\circ\forall{e_2}=\forall{g}\circ\P{g'}$ (Beck-Chevalley),
  and thus:
  $$\begin{array}{rcl}
    \bigsem{\dbigvee\{\sigma_x:x\in f^{-1}(y)\}}_{y\in Y}
    &=&(\P{h}\circ\exists{e_2})(\sem{e_1}_E)
    ~=~(\exists{g}\circ\P{g'})(\sem{e_1}_E)\\[3pt]
    \bigsem{\dbigwedge\{\sigma_x:x\in f^{-1}(y)\}}_{y\in Y}
    &=&(\P{h}\circ\forall{e_2})(\sem{e_1}_E)
    ~=~(\forall{g}\circ\P{g'})(\sem{e_1}_E)\\
  \end{array}$$
  Now we consider the map $q:X\to G$ defined by $q(x)=(\sigma_x,f(x))$
  for all $x\in X$.
  Since~$q$ is surjective, it has a right inverse by (AC), hence we
  have $\exists{q}\circ\P{q}=\forall{q}\circ\P{q}=\id_{\P{G}}$ by
  Lemma~\ref{l:SimplAdj}~(2).
  Therefore:
  $$\begin{array}{r@{~~}c@{~~}l}
    \bigsem{\dbigvee\{\sigma_x:x\in f^{-1}(y)\}}_{y\in Y}
    &=&(\exists{g}\circ\P{g'})(\sem{e_1}_E)
    ~=~(\exists{g}\circ\exists{q}\circ\P{q}\circ\P{g'})(\sem{e_1}_E)\\
    &=&\bigl(\exists(g\circ q)\circ\P(g'\circ q)\bigr)(\sem{e_1}_E)
    ~=~\exists{f}\bigl(\P(g'\circ q)(\sem{e_1}_E)\bigr) \\
    &=&\exists{f}\bigl(\sem{e_1\circ g'\circ q}_X\bigr)
    ~=~\exists{f}(\sem{\sigma}_X) \\
  \end{array}$$
  (since $g\circ q=f$ and $e_1\circ g'\circ q=\sigma$).
  And similarly for ${\forall}$.
\end{proof}

\subsection{Defining the `filter' $\Phi$}
\label{ss:DefFilterSigma}

In Sections~\ref{ss:DefConnSigma} and~\ref{ss:DefQuantSigma}, we
introduced codes
$({\dand}),({\dor}),(\dimp)\in\Sigma^{\Sigma\times\Sigma}$ and
$\dbigvee,\dbigwedge\in\Sigma^{\Pow(\Sigma)}$ such that for all
sets $X$ and for all predicates $p,q\in\P{X}$:
\begin{itemize}
\item If $\sigma,\tau\in\Sigma^X$ are codes for $p,q\in\P{X}$,
  respectively, then:
  $$\begin{array}{l@{\quad}l@{\qquad}l@{\qquad}l@{\quad}l}
    (\sigma_x\dand\tau_x)_{x\in X}&({\in}~\Sigma^X)
    &\text{is a code for}&p\land q&({\in}~\P{X}) \\[3pt]
    (\sigma_x\dor\tau_x)_{x\in X}&({\in}~\Sigma^X)
    &\text{is a code for}&p\lor q&({\in}~\P{X}) \\[3pt]
    (\sigma_x\dimp\tau_x)_{x\in X}&({\in}~\Sigma^X)
    &\text{is a code for}&p\to q&({\in}~\P{X}) \\
  \end{array}$$
\item If $\sigma\in\Sigma^X$ is a code for $p\in\P{X}$
  and if $f:X\to Y$ is any map, then:
  $$\begin{array}{l@{\quad}l@{\qquad}l@{\qquad}l@{\quad}l}
    \Bigl({\textstyle\dbigvee}
    \bigl\{\sigma_x:x\in f^{-1}(y)\bigr\}\Bigr)_{y\in Y}&({\in}~\Sigma^Y)
    &\text{is a code for}&\exists{f}(p)&({\in}~\P{Y}) \\[6pt]
    \Bigl({\textstyle\dbigwedge}
    \bigl\{\sigma_x:x\in f^{-1}(y)\bigr\}\Bigr)_{y\in Y}&({\in}~\Sigma^Y)
    &\text{is a code for}&\forall{f}(p)&({\in}~\P{Y}) \\
  \end{array}$$
\end{itemize}
It now remains to characterize the ordering on each Heyting algebra
$\P{X}$ from the above constructions in~$\Sigma$.
For that, we consider the set $\Phi\subseteq\Sigma$ defined by
$\Phi:=\{\xi\in\Sigma:\sem{\xi}_{\_\in1}=\top_1\}$, writing~$\top_1$
the top element of $\P{1}$.
We check that:
\begin{proposition}\label{p:OrderPhi}
  For all $X\in\Set$ and $\sigma,\tau\in\Sigma^X$, we have
  $$\sem{\sigma}_X\le\sem{\tau}_X\qquad\text{iff}\qquad
  {\textstyle\dbigwedge}\{\sigma_x\dimp\tau_x:x\in X\}\in\Phi\,.$$
\end{proposition}

\begin{proof}
  Writing $1_X:X\to 1$ the unique map from~$X$ to~$1$, we have:
  $$\begin{array}[b]{r@{\quad}c@{\quad}l}
    \sem{\sigma}_X\le\sem{\tau}_X
    &\text{iff}&\top_X\le\sem{\sigma}_X\to\sem{\tau}_X \\
    &\text{iff}&\P{1_X}(\top_1)\le\sem{\sigma_x\dimp\tau_x}_{x\in X}\\
    &\text{iff}&\top_1\le\forall{1_X}
    \bigl(\sem{\sigma_x\dimp\tau_x}_{x\in X}\bigr)\\
    &\text{iff}&\top_1\le\bigsem{\dbigwedge
      \{\sigma_x\dimp\tau_x:x\in X\}}_{\_\in 1}\\
    &\text{iff}&\dbigwedge\{\sigma_x\dimp\tau_x:x\in X\}\in\Phi\\
  \end{array}\eqno\mbox{\qedhere}$$
\end{proof}

So that we can complete the above correspondence between the
operations on the Heyting algebra $\P{X}$ and the analogous
operations on the set~$\Sigma$ by concluding that:
\begin{itemize}
\item If $\sigma,\tau\in\Sigma^X$ are codes for $p,q\in\P{X}$,
  respectively, then:
  $$p\le q\quad({\in}~\P{X})\qquad\text{iff}\qquad
  \mathop{\dbigwedge}\limits_{x\in X}
  (\sigma_x\dimp\tau_x)~\in~\Phi
  \quad({\subseteq}~\Sigma)$$
\end{itemize}

\begin{remark}
  At this stage, it would be tempting to see the set $\Sigma$ as a
  complete Heyting algebra whose structure is given by the operations
  ${\dand},{\dor},{\dimp},\dbigvee,\dbigwedge$, while seeing the
  subset $\Phi\subseteq\Sigma$ as a particular filter of~$\Sigma$.
  Alas, the above operations come with absolutely no algebraic
  property, since in general we have
  $$\xi\dand\xi\neq\xi,\qquad
  \xi\dand\xi'\neq\xi'\dand\xi,\qquad
  \xi\dimp\xi\neq\dtop,\qquad
  {\textstyle\dbigvee}\{\xi\}\neq
  {\textstyle\dbigwedge}\{\xi\}\neq
  \xi,\qquad\text{etc.}$$
  So that in the end, these operations fail to endow the set~$\Sigma$
  with the structure of a (complete) Heyting algebra---although they
  are sufficient in practice to characterize the whole structure of
  the tripos $\P:\Set^{\op}\to\HA$ via the pseudo-filter
  $\Phi\subseteq\Sigma$.
  However, we shall see in Section~\ref{s:ImpAlg} that the set
  $\Sigma$ generates an implicative algebra~$\A$ whose operations
  (arbitrary meets and implication) capture the whole structure of the
  tripos~$\P$ in a more natural way.
\end{remark}

\subsection{Miscellaneous properties}\label{ss:MiscProp}
In this section, we present some properties of the set~$\Sigma$ that
will be useful to construct the implicative algebra~$\A$.
\begin{proposition}[Merging quantifications]%
  \label{p:MergeQuantSigma}
  $$\begin{array}{r@{~{}~}c@{~{}~}l}
    \bigsem{{\dbigvee}\bigl\{{\dbigvee}s:s\in S\bigr\}}
    _{S\in\Pow(\Pow(\Sigma))}&=&
    \bigsem{{\dbigvee}\bigl({\bigcup}S\bigr)}
    _{S\in\Pow(\Pow(\Sigma))}\\[3pt]
    \bigsem{{\dbigwedge}\bigl\{{\dbigwedge}s:s\in S\bigr\}}
    _{S\in\Pow(\Pow(\Sigma))}&=&
    \bigsem{{\dbigwedge}\bigl({\bigcup}S\bigr)}
    _{S\in\Pow(\Pow(\Sigma))}\\
  \end{array}$$
\end{proposition}

\begin{proof}
  Let us consider:
  \begin{itemize}
  \item The membership relation
    $E:=\{(\xi,s):\xi\in s\}\subseteq\Sigma\times\Pow(\Sigma)$
    equipped with the projections $e_1:E\to\Sigma$ and
    $e_2:E\to\Pow(\Sigma)$ (see Section~\ref{ss:DefQuantSigma} above);
  \item The membership relation
    $F:=\{(s,S):s\in S\}\subseteq\Pow(\Sigma)\times\Pow(\Pow(\Sigma))$
    equipped with the projections $f_1:F\to\Pow(\Sigma)$
    and $f_2:F\to\Pow(\Pow(\Sigma))$;
  \item The composed membership relation
    $G:=\{(\xi,s,S):\xi\in s\in S\}\subseteq
    \Sigma\times\Pow(\Sigma)\times\Pow(\Pow(\Sigma))$
    equipped with the functions $g_1:G\to E$ and $g_2:G\to F$
    defined by $g_1(\xi,s,S)=(\xi,s)$ and $g_2(\xi,s,S)=(s,S)$ for all
    $(\xi,s,S)\in G$.
  \end{itemize}
  By construction, we have the following pullback:
  $$\xymatrix{
    G\pullback{6pt}\ar[d]_{g_1}\ar[r]^{g_2}&F\ar[r]^{f_2}\ar[d]^{f_1}&
    \Pow\rlap{$(\Pow(\Sigma))$} \\
    E\ar[d]_{e_1}\ar[r]_{e_2}&\Pow(\Sigma) \\
    \Sigma
  }$$
  hence $\P{f_1}\circ\exists{e_2}=\exists{g_2}\circ\P{g_1}$ and
  $\P{f_1}\circ\forall{e_2}=\forall{g_2}\circ\P{g_1}$
  (Beck-Chevalley).
  Therefore:
  $$\begin{array}[b]{@{}r@{~{}~}c@{~{}~}l@{\hskip-10pt}}
    \bigsem{\dbigvee\bigl\{{\dbigvee}s:s\in S\}}
    _{S\in\Pow(\Pow(\Sigma))}
    &=&\bigsem{\dbigvee\bigl\{{\dbigvee}f_1(z):
      z\in f_2^{-1}(S)\bigr\}}_{S\in\Pow(\Pow(\Sigma))}
    ~=~\exists{f_2}\bigl(\sem{\dbigvee\circ f_1}_{F}\bigr)\\
    &=&(\exists{f_2}\circ\P{f_1})
    \bigl(\sem{\dbigvee}_{\Pow(\Sigma)}\bigr)
    ~=~(\exists{f_2}\circ\P{f_1}\circ\exists{e_2})
    \bigl(\sem{e_1}_E\bigr)\\
    &=&(\exists{f_2}\circ\exists{g_2}\circ\P{g_1})
    \bigl(\sem{e_1}_E\bigr)
    ~=~\exists(f_2\circ g_2)\bigl(\sem{e_1\circ g_1}_G\bigr)\\
    &=&\bigsem{\dbigvee\{(e_1\circ g_1)(z):
      z\in(f_2\circ g_2)^{-1}(S)\}}_{S\in\Pow(\Pow(\Sigma))}\\
    &=&\bigsem{{\dbigvee}\bigl(\bigcup{S}\bigr)}_{S\in\Pow(\Pow(\Sigma))}\\
  \end{array}$$
  And similarly for $\forall$.
\end{proof}

The following proposition expresses that the codes for universal
quantification and implication fulfill the usual property of
distributivity (on the right-hand side of implication):
\begin{proposition}\label{p:DistrForallImp}
  We have:\quad
  $\bigsem{{\dbigwedge}\{\theta\dimp\xi~:~\xi\in s\}}
  _{(\theta,s)\in\Sigma\times\Pow(\Sigma)}
  ~=~\bigsem{\theta\dimp{\dbigwedge}s}
  _{(\theta,s)\in\Sigma\times\Pow(\Sigma)}$.
\end{proposition}

\begin{proof}
  Let us consider the set
  $G:=\{(\theta,\xi,s):\xi\in s\}\subseteq
  \Sigma\times\Sigma\times\Pow(\Sigma)$
  with the corresponding projections
  $g_1,g_2:G\to\Sigma$ and $g_3:G\to\Pow(\Sigma)$.
  We also write $\pi:\Sigma\times\Pow(\Sigma)\to\Sigma$
  the first projection from $\Sigma\times\Pow(\Sigma)$ to~$\Sigma$.
  For all $p\in\P(\Sigma\times\Pow(\Sigma))$, we observe that:
  $$\begin{array}[b]{l@{\quad}l}
    &p~\le~\bigsem{{\dbigwedge}\{\theta\dimp\xi~:~\xi\in s\}}
    _{(\theta,s)\in\Sigma\times\Pow(\Sigma)}\\
    \text{iff}&p~\le~\bigsem{{\dbigwedge}\bigl\{
      g_1(z)\dimp g_2(z)~:~z\in\<g_1,g_3\>^{-1}(\theta,s)
      \bigr\}}_{(\theta,s)\in\Sigma\times\Pow(\Sigma)}\\
    \text{iff}&p~\le~\forall\<g_1,g_3\>\bigl(
    \sem{g_1}_G\to\sem{g_2}_G\bigr) \\
    \text{iff}&\P\<g_1,g_3\>(p)~\le~\sem{g_1}_G\to\sem{g_2}_G \\
    \text{iff}&\P\<g_1,g_3\>(p)\land\sem{g_1}_G~\le~\sem{g_2}_G \\
    \text{iff}&\P\<g_1,g_3\>(p)\land\sem{\pi\circ\<g_1,g_3\>}_G
    ~\le~\sem{g_2}_G \\
    \text{iff}&\P\<g_1,g_3\>\bigl(p\land
    \sem{\pi}_{\Sigma\times\Pow(\Sigma)}\bigr)~\le~\sem{g_2}_G \\
    \text{iff}&p\land\sem{\pi}_{\Sigma\times\Pow(\Sigma)}
    ~\le~\forall\<g_1,g_3\>(\sem{g_2}_G) \\
    \text{iff}&p~\le~
    \sem{\pi}_{\Sigma\times\Pow(\Sigma)}\to
    \forall\<g_1,g_3\>(\sem{g_2}_G)\\
    \text{iff}&p~\le~
    \sem{\theta}_{(\theta,s)\in\Sigma\times\Pow(\Sigma)}\to
    \bigsem{{\dbigwedge}\bigl\{g_2(z)~:~
      z\in\<g_1,g_3\>^{-1}(\theta,s)
      \bigr\}}_{(\theta,s)\in\Sigma\times\Pow(\Sigma)}\\
    \text{iff}&p~\le~
    \sem{\theta}_{(\theta,s)\in\Sigma\times\Pow(\Sigma)}\to
    \bigsem{{\dbigwedge}\bigl\{\xi~:~\xi\in s
      \bigr\}}_{(\theta,s)\in\Sigma\times\Pow(\Sigma)}\\
    \text{iff}&p~\le~
    \bigsem{\theta\dimp{\dbigwedge}s}
    _{(\theta,s)\in\Sigma\times\Pow(\Sigma)}\\
  \end{array}$$
  From which we get the desired equality.
\end{proof}

\begin{corollary}\label{c:DistrForallImp}
  Given a set $X$ and two families $\sigma\in\Sigma^X$ and
  $t\in\Pow(\Sigma)^X$, we have:
  $$\textstyle
  \bigsem{{\dbigwedge}\{\sigma_x\dimp\xi~:~\xi\in t_x\}}_{x\in X}
  ~=~\bigsem{\sigma_x\dimp{\dbigwedge}t_x}_{x\in X}$$
\end{corollary}

\begin{proof}
  Indeed, we have
  $$\begin{array}[b]{r@{~{}~}c@{~{}~}l}
    \bigsem{{\dbigwedge}\{\sigma_x\dimp\xi~:~\xi\in t_x\}}_{x\in X}
    &=&\P(\<\sigma,t\>)\bigl(
    \bigsem{{\dbigwedge}\{\theta\dimp\xi~:~\xi\in s\}}
    _{(\theta,s)\in\Sigma\times\Pow(\Sigma)}\bigr) \\      
    &=&\P(\<\sigma,t\>)\bigl(
    \bigsem{\theta\dimp{\dbigwedge}s}
    _{(\theta,s)\in\Sigma\times\Pow(\Sigma)}\bigr) \\
    &=&\bigsem{\sigma_x\dimp{\dbigwedge}t_x}_{x\in X}\\
  \end{array}\eqno\begin{tabular}[b]{r@{}}
  (by Prop.~\ref{p:DistrForallImp})\\
  \mbox{\qedhere}\\
  \end{tabular}$$
\end{proof}

From now on, we consider the inclusion relation
$F:=\{(s,s'):s\subseteq s'\}\subseteq
\Pow(\Sigma)\times\Pow(\Sigma)$ together with the corresponding
projections $f_1:F\to\Pow(\Sigma)$ and $f_2:F\to\Pow(\Sigma)$.
The following proposition expresses that the operators
${\dbigvee}:\Pow(\Sigma)\to\Sigma$ and
${\dbigwedge}:\Pow(\Sigma)\to\Sigma$ are respectively monotonic and
antitonic w.r.t.\ the domain of quantification:
\begin{proposition}\label{p:SubQuantSigma}
  $\bigsem{{\dbigvee}\circ f_1}_F~\le~
  \bigsem{{\dbigvee}\circ f_2}_F$\ \ and\ \
  $\bigsem{{\dbigwedge}\circ f_1}_F~\ge~
  \bigsem{{\dbigwedge}\circ f_2}_F$.
\end{proposition}

\begin{proof}
  Let us consider the set
  $G:=\{(\xi,\xi',(s,s')):\xi\in s,~\xi'\in s',~(s,s')\in F\}
  \subseteq\Sigma\times\Sigma\times F$
  equipped with the three projections $g_1,g_2:G\to\Sigma$ and
  $g_3:G\to F$.
  We have
  $$\begin{array}{r@{~{}~}c@{~{}~}l}
    \bigsem{{\dbigvee}\circ f_1}_F
    &=&\bigsem{{\dbigvee}
      \bigl\{g_1(z):z\in g_3^{-1}(s,s')\bigr\}}_{(s,s')\in F}
    ~=~\exists{g_3}(\sem{g_1}_G) \\
    \bigsem{{\dbigvee}\circ f_2}_F
    &=&\bigsem{{\dbigvee}
      \bigl\{g_2(z):z\in g_3^{-1}(s,s')\bigr\}}_{(s,s')\in F}
    ~=~\exists{g_3}(\sem{g_2}_G) \\
    \bigsem{{\dbigwedge}\circ f_1}_F
    &=&\bigsem{{\dbigwedge}
      \bigl\{g_1(z):z\in g_3^{-1}(s,s')\bigr\}}_{(s,s')\in F}
    ~=~\forall{g_3}(\sem{g_1}_G) \\
    \bigsem{{\dbigwedge}\circ f_1}_F
    &=&\bigsem{{\dbigwedge}
      \bigl\{g_2(z):z\in g_3^{-1}(s,s')\bigr\}}_{(s,s')\in F}
    ~=~\forall{g_3}(\sem{g_2}_G) \\
  \end{array}$$
  Let us now consider the function $g:G\to G$ defined by
  $g(\xi,\xi',(s,s'))=(\xi,\xi,(s,s'))$ for all
  $(\xi,\xi',(s,s'))\in G$.
  Since $g_2\circ g=g_1$, we have
  $\P{g}(\sem{g_2}_G)=\sem{g_2\circ g}_G=\sem{g_1}_G$ and
  thus
  $$\exists{g}(\sem{g_1}_G)~\le~\sem{g_2}_G
  ~\le~\forall{g}(\sem{g_1}_G)
  \eqno\text{(by left/right adjunction)}$$
  Combining the above inequalities with the fact that
  $g_3\circ g=g_3$, we deduce that:
  $$\begin{array}[b]{l}
    \bigsem{{\dbigvee}\circ f_1}_F
    ~=~\exists{g_3}(\sem{g_1}_G)
    ~=~\exists{g_3}(\exists{g}(\sem{g_1}_G))
    ~\le~\exists{g_3}(\sem{g_2}_G)
    ~=~\bigsem{{\dbigvee}\circ f_2}_F\\
    \bigsem{{\dbigwedge}\circ f_1}_F
    ~=~\forall{g_3}(\sem{g_1}_G)
    ~=~\forall{g_3}(\forall{g}(\sem{g_1}_G))
    ~\ge~\forall{g_3}(\sem{g_2}_G)
    ~=~\bigsem{{\dbigwedge}\circ f_2}_F\,.\\
  \end{array}\eqno\mbox{\qedhere}$$
\end{proof}

\begin{corollary}\label{c:SubQuantSigma}
  Given a set~$X$ and two families $a,b\in\Pow(\Sigma)^X$ such
  that $a_x\subseteq b_x$ for all $x\in X$, we have:
  $\bigsem{{\dbigvee}\circ a}_X~\le~
  \bigsem{{\dbigvee}\circ b}_X$\ \ and\ \
  $\bigsem{{\dbigwedge}\circ a}_X~\ge~
  \bigsem{{\dbigwedge}\circ b}_X$.
\end{corollary}

\begin{proof}
  Let $c:=(a_x,b_x)_{x\in X}\in F^X$.
  From Prop.~\ref{p:SubQuantSigma} we get:
  $$\begin{array}[b]{@{}l@{}}
    \bigsem{{\dbigvee}\circ a}_X
    =\bigsem{{\dbigvee}\circ f_1\circ c}_X
    =\P{c}\bigl(\bigsem{{\dbigvee}\circ f_1}\bigr)
    \le\P{c}\bigl(\bigsem{{\dbigvee}\circ f_2}\bigr)
    =\bigsem{{\dbigvee}\circ f_2\circ c}_X
    =\bigsem{{\dbigvee}\circ b}_X \\
    \bigsem{{\dbigwedge}\circ a}_X
    =\bigsem{{\dbigwedge}\circ f_1\circ c}_X
    =\P{c}\bigl(\bigsem{{\dbigwedge}\circ f_1}\bigr)
    \ge\P{c}\bigl(\bigsem{{\dbigwedge}\circ f_2}\bigr)
    =\bigsem{{\dbigwedge}\circ f_2\circ c}_X
    =\bigsem{{\dbigwedge}\circ b}_X \\
  \end{array}\eqno\mbox{\qedhere}$$
\end{proof}

\section{Extracting the implicative algebra}
\label{s:ImpAlg}

In Section~\ref{ss:DefFilterSigma}, we have seen that the structure of
the tripos $\P:\Set^{\op}\to\HA$ can be fully characterized from the
binary operations
$({\dand}),({\dor}),({\dimp}):\Sigma\times\Sigma\to\Sigma$
and the infinitary operations
$({\dbigvee}),({\dbigwedge}):\Pow(\Sigma)\to\Sigma$ via some subset
$\Phi\subseteq\Sigma$ (the `pseudo-filter').
However, these operations fail to endow the set~$\Sigma$ itself with
the structure of a complete Heyting algebra, for they lack the most
basic algebraic properties that are required by such a structure.

In this section, we shall construct a particular implicative structure
$\A=(\A,{\cle},{\to})$ from the set~$\Sigma$ equipped with the only
operations $({\dimp}):\Sigma\times\Sigma\to\Sigma$ and
$({\dbigwedge}):\Pow(\Sigma)\to\Sigma$, using a construction that is
reminiscent from the construction of graph models~\cite{Eng81}.
As we shall see, the main interest of such a construction is that:
\begin{enumerate}
\item The carrier set $\A$ can be used as an alternative set of
  propositions, for it possesses its own generic predicate
  $\tr_{\A}\in\P\A$.
\item The two operations $({\to}):\A\times\A\to\A$ and
  $({\bigmeet}):\Pow(\A)\to\A$ that naturally come with the
  implicative structure~$\A$ constitute codes (in the sense of the
  new generic predicate $\tr_{\A}\in\P\A$) for implication and
  universal quantification in the tripos~$\P$.
\item The ordering on each Heyting algebra $\P{X}$ (for $X\in\Set$)
  can be characterized from the above two operations via a particular
  separator $S\subseteq\A$.
\end{enumerate}
From the above properties, we shall easily conclude that the
tripos~$\P:\Set^{\op}\to\HA$ is isomorphic to the tripos induced by
the implicative algebra $(\A,{\cle},{\to},S)$.

\subsection{Defining the set $\A_0$ of atoms}
\label{ss:Atoms}
The construction of the implicative structure~$\A$ is achieved in two
steps.
First we construct from the set~$\Sigma$ of propositions a (larger)
set $\A_0$ of \emph{atoms} equipped with a preorder~$\le$.
Then we let $\A:=\Powup(\A_0)$, where $\Powup(\A_0)$ denotes the set
of all upwards closed subsets of~$\A_0$ w.r.t.\ the preorder~$\le$.
Formally, the set $\A_0$ of atoms is inductively defined from the
following two clauses:
\begin{enumerate}
\item If $\xi\in\Sigma$, then $\dot\xi\in\A_0$%
  \footnote{Here, we use the dot notation $\dot\xi$ only to
    distinguish the code $\xi\in\Sigma$ from its image
    $\dot\xi\in\A_0$.}\quad (base case).
\item If $s\in\Pow(\Sigma)$ and $\alpha\in\A_0$, then
  $(s\mapsto\alpha)\in\A_0$\quad (inductive case).
\end{enumerate}
In other words, each atom $\alpha\in\A_0$ is a finite list of subsets
of~$\Sigma$ terminated by an element of~$\Sigma$, that is:\ \
$\alpha=s_1\mapsto\cdots\mapsto s_n\mapsto\dot\xi$\ \
for some $s_1,\ldots,s_n\in\Pow(\Sigma)$ and $\xi\in\Sigma$.
The set $\A_0$ is equipped with the binary relation $\alpha\le\beta$
that is inductively defined from the two rules
$$\infer{\dot\xi\le\dot\xi}{}\qquad\qquad
\infer{s\mapsto\alpha~\le~s'\mapsto\alpha'}{
  s\subseteq s'&\quad \alpha\le\alpha'
}$$

\begin{fact}
  The relation $\alpha\le\beta$ is a preorder on $\A_0$.
\end{fact}

\begin{proof}
  By induction on $\alpha\in\A_0$, we successively prove
  (1) that $\alpha\le\alpha$ and
  (2) that $\alpha\le\beta$ and $\beta\le\gamma$ together imply that
  $\alpha\le\gamma$, for all $\beta,\gamma\in\A_0$.
\end{proof}

Finally, the set $\A_0$ is equipped with a conversion function
$\phi_0:\A_0\to\Sigma$ that is defined by
$$\phi_0(\dot\xi)~:=~\xi\qquad\text{and}\qquad
\phi_0(s\mapsto\alpha)~:=~
\bigl({\textstyle\dbigwedge{s}}\bigr)\dimp\phi_0(\alpha)$$
(Note that by construction, this function is surjective.)

\begin{remark}[Relationship with graph models]
  In the theory of graph models~\cite{Eng81}, the set of atoms~$\A_0$
  would be naturally defined from the grammar 
  $$\alpha,\beta\in\A_0\quad::=\quad
  \dot{\xi}\quad|\quad\{\alpha_1,\ldots,\alpha_n\}\mapsto\beta
  \eqno(\xi\in\Sigma)$$
  that is, as the least solution of the set-theoretic equation
  $\A_0=\Sigma+\Powfin(\A_0)\times\A_0$.
  However, such a construction would yield an \emph{applicative}
  structure upon the set $\A=\Powup(\A_0)$---it would even constitute
  a ($D_{\infty}$-like) model of the $\lambda$-calculus---, but it
  would not yield an \emph{implicative} structure, for the restriction
  to the finite subsets of $\A_0$ in the left-hand side of the
  construction $\{a_1,\ldots,\alpha_n\}\mapsto\beta$ actually prevents
  defining an implication with the desired properties.\par
  To fix this problem, it would be natural to relax the condition of
  finiteness, by considering instead the equation
  $\A_0=\Sigma+\Pow(\A_0)\times\A_0$.
  Alas, such an equation has no solution (for obvious cardinality
  reasons), so that the trick consists to replace arbitrary subsets of
  $\A_0$ (in the left-hand side of the construction $s\mapsto\beta$)
  by arbitrary subsets of~$\Sigma$, using the fact that the subsets
  of~$\A_0$ can always be converted (element-wise) into subsets
  of~$\Sigma$, via the conversion function $\phi_0:\A_0\to\Sigma$.
  So that in the end, we obtain the set-theoretic equation
  $\A_0=\Sigma+\Pow(\Sigma)\times\A_0$, whose least solution is
  precisely the set $\A_0$ we defined above.
\end{remark}

\subsection{Defining the implicative structure $(\A,{\cle},{\to})$}
\label{ss:ImpStruct}
Let $\A:=\Powup(\A_0)$ be the set of upwards closed subsets of~$\A_0$
(w.r.t.\ the preorder $\le$ on~$\A_0$), equipped with the ordering
$a\cle b$ defined by $a\cle b$ iff $a\supseteq b$ (reverse inclusion)
for all $a,b\in\A$.
It is clear that:
\begin{proposition}
  The poset $(\A,{\cle})=(\Powup(\A_0),{\supseteq})$ is a
  complete lattice.
\end{proposition}

Note that in this complete lattice, (finitary or infinitary) meets and
joins are respectively given by unions and intersections.
In particular, we have $\bot_{\A}=\A$ and $\top_{\A}=\varnothing$.

Let $\tphi_0:\A\to\Pow(\Sigma)$ be the function defined by
$\tphi_0(a)=\{\phi_0(\alpha):\alpha\in a\}$ for all $a\in\A$.
For each set of codes $s\in\Pow(\Sigma)$, we write
$s^{\subseteq}:=\{s'\in\Pow(\Sigma):s\subseteq s'\}$.
We now equip the complete lattice $(\A,{\cle})$ with the implication
defined by
$$a\to b~:=~\bigl\{s\mapsto\beta~:~
s\in\tphi_0(a)^{\subseteq},~\beta\in b\bigr\}
\eqno(\text{for all}~a,b\in\A)$$
Note that by construction, we have
$(a\to b)\in\A~({=}~\Powup(\A_0))$ for all $a,b\in\A$.

\begin{proposition}
  The triple $(\A,{\cle},{\to})$ is an implicative structure.
\end{proposition}

\begin{proof}
  \emph{Variance of implication.}\quad
  Let $a,a',b,b'\in\A$ such that $a'\cle a$ and $b\cle b$, that is:
  $a\subseteq a'$ and $b'\subseteq b$.
  We observe that
  $$\begin{array}{r@{~{}~}c@{~{}~}l}
    a'\to b'&=&\bigl\{s\mapsto\beta~:~
    s\in\tphi_0(a')^{\subseteq},~\beta\in b'\bigr\}\\
    &\subseteq&\bigl\{s\mapsto\beta~:~
    s\in\tphi_0(a)^{\subseteq},~\beta\in b\bigr\}
    ~=~a\to b \\
  \end{array}$$
  (since $\tphi_0(a)\subseteq\tphi_0(a')$ and $b'\subseteq b$), which
  means that $(a\to b)\cle(a'\to b')$.\par\noindent
  \emph{Distributivity of meets w.r.t.\ $\to$.}\quad
  Given $a\in\A$ and $B\subseteq\A$, we observe that
  $$\begin{array}[b]{r@{~{}~}c@{~{}~}l}
    a\to\bigmeet\!B
    &=&\ds\bigl\{s\mapsto\beta~:~
    s\in\tphi_0(a)^{\subseteq},~
    \beta\in{\textstyle\bigcup}B\bigr\}\\[6pt]
    &=&\ds\bigcup_{b\in B}\bigl\{s\mapsto\beta~:~
    s\in\tphi_0(a)^{\subseteq},~\beta\in b\bigr\}
    ~=~\bigmeet_{\!\!b\in B\!\!}(a\to b) \\
  \end{array}\eqno\mbox{\qedhere}$$
\end{proof}

\subsection{Viewing $\A$ as a new set of propositions}
Let us now consider the two conversion functions $\phi:\A\to\Sigma$
and $\psi:\Sigma\to\A$ defined by
$$\begin{array}{r@{~{}~}c@{~{}~}l}
  \phi(a)&:=&\dbigwedge\tphi_0(a)
  ~=~\dbigwedge\{\phi_0(\alpha):\alpha\in a\}\\[3pt]
  \psi(\xi)&:=&\{\dot\xi\}\\
\end{array}\eqno\begin{array}{r@{}}
(\text{for all}~a\in\A)\\[3pt]
(\text{for all}~\xi\in\Sigma)\\
\end{array}$$
as well as the predicate
$\tr_{\A}:=\P\phi(\tr_{\Sigma})~({=}~\sem{\phi}_{\A})\in\P\A$.\ \
We easily check that:
\begin{lemma}
  $\P\psi(\tr_{\A})=\tr_{\Sigma}$.
\end{lemma}

\begin{proof}
  Since $\phi(\psi(\xi))={\dbigwedge}\{\xi\}$ for all $\xi\in\Sigma$,
  we have:
  $$\begin{array}[b]{r@{~{}~}c@{~{}~}l}
    \P\psi(\tr_{\A})
    &=&\P\psi(\sem{\phi}_{\A})
    ~=~\sem{\phi\circ\psi}_{\Sigma}
    ~=~\bigsem{\dbigwedge\{\xi\}}_{\xi\in\Sigma}
    ~=~\bigsem{\dbigwedge\bigl\{\id(\xi'):
      \xi'\in\id^{-1}(\xi)\bigr\}}_{\xi\in\Sigma}\\
    &=&\forall\id_{\Sigma}\bigl(\sem{\id_{\Sigma}}_{\Sigma}\bigr)
    ~=~\sem{\id_{\Sigma}}_{\Sigma}
    ~=~\P\id_{\Sigma}(\tr_{\Sigma})~=~\tr_{\Sigma}\,.
  \end{array}\eqno\mbox{\qedhere}$$
\end{proof}
Therefore:
\begin{proposition}
  The predicate $\tr_{\A}\in\P\A$ is a generic predicate for the
  tripos~$\P$.
\end{proposition}

\begin{proof}
  For each set $X$, we want to show that the function
  $\csem{\_}_X:\A^X\to\P{X}$ defined by $\csem{a}_X=\P{a}(\tr_{\A})$
  for all $a\in\A^X$ is surjective.
  For that, we take $p\in\P{X}$ and pick a code $\sigma\in\Sigma^X$
  such that $p=\sem{\sigma}_X=\P\sigma(\tr_{\Sigma})$.
  From the above lemma, we get:
  $$p~=~\P{\sigma}(\tr_{\Sigma})~=~\P{\sigma}(\P\psi(\tr_{\A}))
  ~=~\P(\psi\circ\sigma)(\tr_{\A})~=~\csem{\psi\circ\sigma}_X$$
  hence $a:=\psi\circ\sigma\in\A^{X}$ is a code for $p$ in the sense
  of the predicate~$\tr_{\A}\in\P\A$.
\end{proof}

To sum up, we now have
two sets of propositions~$\Sigma$ and~$\A$,
two generic predicates $\tr_{\Sigma}\in\P\Sigma$ and
$\tr_{\A}\in\P\A$,
as well as two decoding functions $\sem{\_}_X:\Sigma^X\to\P{X}$ and
$\csem{\_}_X:\A^X\to\P{X}$.
As usual, we write $\phi^X:\A^X\to\Sigma^X$ and
$\psi^X:\Sigma^X\to\A^X$ (for $X\in\Set$) the natural transformations
induced by the two maps $\phi:\A\to\Sigma$ and $\psi:\Sigma\to\A$.
We easily check that:
\begin{proposition}
  For each set~$X$, the following two diagrams commute:
  $$\xymatrix{
    \Sigma^X\ar[r]^{\sem{\_}_X}&\P{X}\ar@{=}[d] \\
    \A^X\ar[u]^{\phi^X}\ar[r]_{\csem{\_}_X}&\P{X} \\
  }\qquad\qquad
  \xymatrix{
    \Sigma^X\ar[d]_{\psi^X}\ar[r]^{\sem{\_}_X}&\P{X}\ar@{=}[d] \\
    \A^X\ar[r]_{\csem{\_}_X}&\P{X} \\
  }$$
\end{proposition}

\begin{proof}
  For all $a\in\A^X$, we have\ \
  $\sem{\phi^X(a)}_X=\sem{\phi\circ a}_X
  =\P{a}(\P\phi(\tr_{\Sigma}))=\P{a}(\tr_{\A})=\csem{a}_X$.\\
  And for all $\sigma\in\Sigma^X$, we have\ \
  $\csem{\psi^X(\sigma)}_X=\csem{\psi\circ\sigma}
  =\P{\sigma}(\P\psi(\tr_{\A}))=\P{\sigma}(\tr_{\Sigma})
  =\sem{\sigma}_X$.
\end{proof}

\subsection{Universal quantification in~$\A$}
\label{ss:QuantA}
By analogy with the construction performed in
Section~\ref{ss:DefQuantSigma}, we now consider the membership
relation
$E':=\{(a,A):a\in A\}\subseteq\A\times\Pow(\A)$ together with the
corresponding projections $e'_1:E'\to\A$ and $e'_2:E'\to\Pow(\A)$.

The following proposition states that the operator
$({\bigmeet}):\Pow(\A)\to\A$ is a code for universal quantification in
the sense of the generic predicate $\tr_{\A}\in\P\A$:
\begin{proposition}\label{p:QuantA0}
  $\bigcsem{{\bigmeet}A}_{A\in\Pow(\A)}=
  \forall{e'_2}\bigl(\csem{e'_1}_{E'}\bigr)$.
\end{proposition}

\begin{proof}
  Indeed, we have:
  $$\begin{array}[b]{r@{~{}~}c@{~{}~}l@{\hskip-5mm}}
    \bigcsem{{\bigmeet}A}_{A\in\Pow(\A)}
    &=&\bigcsem{{\bigcup}A}_{A\in\Pow(\A)}
    ~=~\bigsem{\phi\bigl({\bigcup}A\bigr)}_{A\in\Pow(\A)}
    ~=~\bigsem{{\dbigwedge}\tphi_0\bigl({\bigcup}A\bigr)}
    _{A\in\Pow(\A)}\\
    &=&\bigsem{{\dbigwedge}{\bigcup}\Pow\tphi_0(A)}_{A\in\Pow(\A)}
    ~=~\P(\Pow\tphi_0)\bigl(\bigsem{{\dbigwedge}{\bigcup}S}
    _{S\in\Pow(\Pow(\Sigma))}\bigr) \\
    &=&\P(\Pow\tphi_0)\bigl(\bigsem{{\dbigwedge}
      \bigl\{{\dbigwedge}s:s\in S\bigr\}}
    _{S\in\Pow(\Pow(\Sigma))}\bigr)\\
    &=&\bigsem{{\dbigwedge}
      \bigl\{{\dbigwedge}\tphi_0(a):a\in A\bigr\}}_{A\in\Pow(\A)}
    ~=~\bigsem{{\dbigwedge}
      \bigl\{\phi(a):a\in A\bigr\}}_{A\in\Pow(\A)}\\
    &=&\bigsem{{\dbigwedge}
      \bigl\{\phi(e'_1(p)):p\in{e'}_2^{-1}(A)\}}_{A\in\Pow(\A)}\\
    &=&\forall{e'_2}\bigl(\sem{\phi\circ e'_1}_{E'}\bigr)
    ~=~\forall{e'_2}\bigl(\csem{e'_1}_{E'}\bigr)\,. \\
  \end{array}\eqno\begin{tabular}[b]{r@{}}
  (by Prop.~\ref{p:MergeQuantSigma})\\\\\\
  \mbox{\qedhere}\\
  \end{tabular}$$
\end{proof}

From the above result, we deduce that:
\begin{proposition}\label{p:QuantA}
  Given a code $a\in\A^X$ and a map $f:X\to Y$, we have:
  $$\textstyle\bigcsem{\bigmeet\{a_x:x\in f^{-1}(y)\}}_{y\in Y}
  ~=~\forall{f}(\csem{a}_X)\eqno({\in}~\P{Y})$$
\end{proposition}

\begin{proof}
  Same argument as for Prop.~\ref{p:QuantSigma}
  p.~\pageref{p:QuantSigma}.
\end{proof}

\subsection{Implication in~$\A$}
\label{ss:ImpA}
It now remains to show that the operation $({\to}):\A\times\A\to\A$ is
a code for implication in the sense of the generic predicate
$\tr_{\A}\in\P\A$.
For that, we first need to prove the following technical lemma:
\begin{lemma}\label{l:ImpTech}
  $\Bigsem{{\dbigwedge}
    \Bigl\{\bigl({\dbigwedge}s'\bigr)\dimp\xi~:~
    s'\in s^{\subseteq},~\xi\in t\Bigr\}}
  _{(s,t)\in\Pow(\Sigma)^2}=
  \Bigsem{{\dbigwedge}
    \Bigl\{\bigl({\dbigwedge}s\bigr)\dimp\xi~:~\xi\in t\Bigr\}}
  _{(s,t)\in\Pow(\Sigma)^2}$.
\end{lemma}

\begin{proof}
  Let us consider the set $G:=\{(s,t,s',\xi):s'\supseteq s,~\xi\in t\}
  \subseteq\Pow(\Sigma)\times\Pow(\Sigma)\times\Pow(\Sigma)\times\Sigma$
  equipped with the four projections
  $g_1,g_2,g_3:G\to\Pow(\Sigma)$, $g_4:G\to\Sigma$ and the function
  $g:G\to G$ defined by $g(s,t,s',\xi)=g(s,t,s,\xi)$ for all
  $(s,t,s',\xi)\in G$.
  We observe that
  $$\begin{array}{l@{~{}~}l}
    &\Bigsem{{\dbigwedge}
      \Bigl\{\bigl({\dbigwedge}s'\bigr)\dimp\xi~:~
      s'\in s^{\subseteq},~\xi\in t\Bigr\}}
    _{(s,t)\in\Pow(\Sigma)^2}\\
    =&\Bigsem{{\dbigwedge}
      \Bigl\{\bigl({\dbigwedge}g_3(z)\bigr)\dimp g_4(z)~:~
      z\in\<g_1,g_2\>^{-1}(s,t)\Bigr\}}_{(s,t)\in\Pow(\Sigma)^2}\\
    =&\forall\<g_1,g_2\>\Bigl(
    \bigsem{{\dbigwedge}\circ g_3}_G\to\sem{g_4}_G\Bigr)\\
  \end{array}$$
  whereas
  $$\begin{array}{l@{~{}~}l}
    &\Bigsem{{\dbigwedge}
      \Bigl\{\bigl({\dbigwedge}s\bigr)\dimp\xi~:~\xi\in t\Bigr\}}
    _{(s,t)\in\Pow(\Sigma)^2}\\
    =&\Bigsem{{\dbigwedge}
      \Bigl\{\bigl({\dbigwedge}g_1(z)\bigr)\dimp g_4(z)~:~
      z\in\<g_1,g_2\>^{-1}(s,t)\Bigr\}}_{(s,t)\in\Pow(\Sigma)^2}\\
    =&\forall\<g_1,g_2\>\Bigl(
    \bigsem{{\dbigwedge}\circ g_1}_G\to\sem{g_4}_G\Bigr)\\
  \end{array}$$
  So that we have to prove that\ \
  $\forall\<g_1,g_2\>\Bigl(
  \bigsem{{\dbigwedge}\circ g_3}_G\to\sem{g_4}_G\Bigr)~=~
  \forall\<g_1,g_2\>\Bigl(
  \bigsem{{\dbigwedge}\circ g_1}_G\to\sem{g_4}_G\Bigr)$.
  \smallbreak\noindent
  $({\le})$\quad We observe that
  $$\textstyle
  \P{g}\Bigl(\bigsem{{\dbigwedge}\circ g_3}_G\to\sem{g_4}_G\Bigr)
  ~=~\bigsem{{\dbigwedge}\circ g_3\circ g}_G\to\sem{g_4\circ g}_G
  ~=~\bigsem{{\dbigwedge}\circ g_1}_G\to\sem{g_4}_F\,,$$
  since $g_3\circ g=g_1$ and $g_4\circ g=g_4$.
  By right adjunction we thus get
  $$\begin{array}[b]{r@{~{}~}c@{~{}~}l}
    \bigsem{{\dbigwedge}\circ g_3}_G\to\sem{g_4}_G&\le&
    \forall{g}\Bigl(
    \bigsem{{\dbigwedge}\circ g_1}_G\to\sem{g_4}_G\Bigr) \\
    \forall\<g_1,g_2\>\Bigl(
    \bigsem{{\dbigwedge}\circ g_3}_G\to\sem{g_4}_G\Bigr)&\le&
    (\forall\<g_1,g_2\>\circ\forall{g})\Bigl(
    \bigsem{{\dbigwedge}\circ g_1}_G\to\sem{g_4}_G\Bigr)\\
    \forall\<g_1,g_2\>\Bigl(
    \bigsem{{\dbigwedge}\circ g_3}_G\to\sem{g_4}_G\Bigr)&\le&
    \forall\<g_1,g_2\>\Bigl(
    \bigsem{{\dbigwedge}\circ g_1}_G\to\sem{g_4}_G\Bigr)\,,\\
  \end{array}\leqno\begin{tabular}[b]{@{}l}
    hence\\[3pt]and thus\\
  \end{tabular}$$
  using the fact that
  $\<g_1,g_2\>\circ g=\<g_1\circ g,g_2\circ g\>=\<g_1,g_2\>$.
  \smallbreak\noindent
  $({\ge})$\quad From Coro.~\ref{c:SubQuantSigma}, we get
  $\bigsem{{\dbigwedge}\circ g_3}_G\le
  \bigsem{{\dbigwedge}\circ g_1}_G$
  (since $g_1(z)\subseteq g_3(z)$ for all $z\in G$).
  Hence\ \
  $\bigsem{{\dbigwedge}\circ g_3}_G\to\sem{g_4}_G~\ge~
  \bigsem{{\dbigwedge}\circ g_1}_G\to\sem{g_4}_G$,\ \
  and thus:
  $$\textstyle\forall\<g_1,g_2\>\Bigl(
  \bigsem{{\dbigwedge}\circ g_3}_G\to\sem{g_4}_G\Bigr)~\ge~
  \forall\<g_1,g_2\>\Bigl(
  \bigsem{{\dbigwedge}\circ g_1}_G\to\sem{g_4}_G\Bigr)\,.
  \eqno\mbox{\qedhere}$$
\end{proof}

We can now state the desired property:
\begin{proposition}\label{p:ImpA0}
  $\csem{a\to b}_{(a,b)\in\A^2}=
  \csem{\pi}_{\A^2}\to\csem{\pi'}_{\A^2}
  ~({\in}~\P(\A\times\A))$,\\
  writing $\pi,\pi':\A^2\to\A$ the two projections from~$\A^2$
  to~$\A$.
\end{proposition}

\begin{proof}
  Indeed, we have:
  $$\begin{array}[b]{r@{~{}~}l@{~{}~}l@{\hskip-10mm}}
    \csem{a\to b}_{(a,b)\in\A^2}
    &=&\sem{\phi(a\to b)}_{(a,b)\in\A^2} \\
    &=&\Bigsem{\phi\Bigl(\bigl\{s'\mapsto\beta~:~
      s'\in\tphi_0(a)^{\subseteq},~\beta\in b\bigr\}\Bigr)}
    _{(a,b)\in\A^2} \\
    &=&\Bigsem{{\dbigwedge}
      \Bigl\{\bigl({\dbigwedge}s'\bigr)\dimp\xi~:~
      s'\in\tphi_0(a)^{\subseteq},~\xi\in\tphi_0(b)\Bigr\}}
    _{(a,b)\in\A^2} \\
    &=&\P(\tphi_0\times\tphi_0)\left(
    \Bigsem{{\dbigwedge}
      \Bigl\{\bigl({\dbigwedge}s'\bigr)\dimp\xi~:~
      s'\in s^{\subseteq},~\xi\in t\Bigr\}}
    _{(s,t)\in\Pow(\Sigma)^2}\right) \\
    &=&\P(\tphi_0\times\tphi_0)\left(
    \Bigsem{{\dbigwedge}
      \Bigl\{\bigl({\dbigwedge}s\bigr)\dimp\xi~:~\xi\in t\Bigr\}}
    _{(s,t)\in\Pow(\Sigma)^2}\right) \\
    &=&\P(\tphi_0\times\tphi_0)\left(
    \bigsem{\bigl({\dbigwedge}s\bigr)\dimp\bigl({\dbigwedge}t\bigr)}
    _{(s,t)\in\Pow(\Sigma)^2}\right) \\
    &=&\bigsem{\bigl({\dbigwedge}\tphi_0(a)\bigr)\dimp
      \bigl({\dbigwedge}\tphi_0(b)\bigr)}_{(a,b)\in\A^2}
    ~=~\sem{\phi(a)\dimp\phi(b)}_{(a,b)\in\A^2}\\
    &=&\sem{\phi\circ\pi}_{\A^2}\to\sem{\phi\circ\pi'}_{\A^2}
    ~=~\csem{\pi}_{\A^2}\to\csem{\pi'}_{\A^2} \\
  \end{array}\eqno\begin{tabular}[b]{r@{}}
  (by Lemma~\ref{l:ImpTech})\\[5pt]
  (by Coro.~\ref{c:DistrForallImp})\\[3pt]\\
  \mbox{\qedhere}\\
  \end{tabular}$$
\end{proof}

\begin{proposition}\label{p:ImpA}
  Let $X$ be a set.
  For all codes $a,b\in\A^X$, we have
  $$\textstyle
  \csem{a_x\to b_x}_{x\in X}~=~\csem{a}_X\to\csem{b}_X
  \eqno({\in}~\P{X})$$
\end{proposition}

\begin{proof}
  Same argument as for Prop.~\ref{p:ConnSigma}
  p.~\pageref{p:ConnSigma}.
\end{proof}

\subsection{Defining the separator $S\subseteq\A$}
By analogy with the construction of the pseudo-filter
$\Phi\subseteq\Sigma$ (cf Section~\ref{ss:DefFilterSigma}), we let
$$S~:=~\bigl\{a\in\A~:~\csem{a}_{\_\in1}=\top_1\bigr\}
\eqno({\subseteq}~\A)$$
writing $\top_1$ the top element of $\P{1}$.
Note that by construction, we have
$$S~=~\bigl\{a\in\A~:~\sem{\phi(a)}_{\_\in1}=\top_1\bigr\}
~=~\bigl\{a\in\A~:~\phi(a)\in\Phi\bigr\}
~=~\phi^{-1}(\Phi)\,.$$

\begin{proposition}
  The subset $S\subseteq\A$ is a separator of the implicative
  structure $(\A,{\cle},{\to})$.
\end{proposition}

\begin{proof}
  \emph{$S$ is upwards closed.}\quad
  Let $a,b\in\A$ such that $a\in S$ and $a\cle b$ (that is:
  $b\subseteq a$).
  From these assumptions, we have $\csem{a}_{\_\in1}=\top_1$
  and $\tphi_0(b)\subseteq\tphi_0(a)$, hence
  $$\textstyle\top_1~=~\csem{a}_{\_\in 1}
  ~=~\bigsem{{\dbigwedge}\tphi_0(a)}_{\_\in 1}
  ~=~\bigsem{{\dbigwedge}\circ(\tphi_0(a))_{\_\in 1}}_1
  ~\le~\bigsem{{\dbigwedge}\circ(\tphi_0(b))_{\_\in 1}}_1
  ~=~\csem{b}_{\_\in 1}$$
  (from Coro.~\ref{c:SubQuantSigma}) and thus
  $\csem{b}_{\_\in 1}=\top_1$.
  Therefore $b\in S$.
  \smallbreak\noindent
  \emph{$S$ contains $\mathbf{K}^{\A}$ and $\mathbf{S}^{\A}$}.\quad
  We observe that
  $$\begin{array}{r@{~{}~}c@{~{}~}l}
    \bigcsem{\mathbf{K}^{\A}}_{\_\in1}
    &=&\Bigcsem{\bigmeet_{a\in\A}\bigmeet_{b\in\B}
      (a\to b\to a)}_{\_\in1}\\[6pt]
    &=&\forall{\pi_{1,\A}}\Bigl(
    \Bigcsem{\bigmeet_{b\in\B}(a\to b\to a)}
    _{(\_\,,a)\in1\times\A}\Bigr)\\[6pt]
    &=&\forall{\pi_{1,\A}}\bigl(\forall{\pi_{1\times\A,\A}}\bigl(
    \bigcsem{a\to b\to a}
    _{((\_\,,a),b)\in(1\times\A)\times\A}\bigr)\bigr)\\[3pt]
    &=&\forall{\pi_{1,\A}}\bigl(\forall{\pi_{1\times\A,\A}}
    \bigl(\top_{(1\times\A)\times\A}\bigr)\bigr)
    ~=~\top_1\\
  \end{array}$$
  hence $\mathbf{K}^{\A}\in S$.
  The proof that  $\mathbf{S}^{\A}\in S$ is analogous.
  \smallbreak\noindent
  \emph{$S$ is closed under modus ponens}.\quad
  Suppose that $(a\to b)\in S$ and $a\in S$.
  This means that
  $\csem{a\to b}_{\_\in1}=
  \csem{a}_{\_\in 1}\to\csem{b}_{\_\in 1}=\top_1$ and
  $\csem{a}_{\_\in1}=\top_1$.
  Hence $\csem{b}_{\_\in1}=\top_1$, and thus $b\in S$.
\end{proof}

Similarly to Prop.~\ref{p:OrderPhi} p.~\pageref{p:OrderPhi}, the
following proposition characterizes the ordering on each Heyting
algebra $\P{X}$ from the operations of~$\A$ and the separator
$S\subseteq\A$:
\begin{proposition}\label{p:CharacSeparA}
  For all sets $X$ and for all codes $a,b\in\A^X$, we have:
  $$\csem{a}_X\le\csem{b}_X\quad\text{iff}\quad
  \bigmeet_{x\in X}(a_x\to b_x)~\in~S\,.$$
\end{proposition}

\begin{proof}
  Writing $1_X:X\to 1$ the unique map from~$X$ to~$1$, we have:
  $$\begin{array}[b]{r@{\quad}c@{\quad}l}
    \csem{a}_X\le\csem{b}_X
    &\text{iff}&\top_X\le\csem{a}_X\to\csem{b}_X \\
    &\text{iff}&\P{1_X}(\top_1)\le\csem{a_x\to b_x}_{x\in X}\\
    &\text{iff}&\top_1\le\forall{1_X}
    \bigl(\csem{a_x\to b_x}_{x\in X}\bigr)\\
    &\text{iff}&\top_1\le\bigcsem{{\bigmeet}
      \{a_x\to b_x:x\in X\}}_{\_\in 1}\\
    &\text{iff}&\ds\bigmeet_{x\in X}(a_x\to b_x)~\in~S\,. \\
  \end{array}\eqno\mbox{\qedhere}$$
\end{proof}

\subsection{Constructing the isomorphism}
Let us now write $\P':\Set^{\op}\to\HA$ the tripos induced by the
implicative algebra $(\A,{\cle},{\to},S)$ (Section~\ref{ss:ImpAlg}).
Recall that for each set $X$, the Heyting algebra $\P'{X}:=\A^X/S[X]$
is the poset reflection of the preordered set $(\A^X,\ent_{S[X]})$
whose preorder $\ent_{S[X]}$ is given by
$$a\ent_{S[X]}b\quad\text{iff}\quad
\bigmeet_{x\in X}(a_x\to b_x)~\in~S
\eqno(\text{for all}~a,b\in\A^X)$$
It now remains to show that:
\begin{proposition}
  The implicative tripos~$\P'$ is isomorphic to the tripos~$\P$.
\end{proposition}

\begin{proof}
  Let us consider the family of maps
  $\rho_X:=\csem{\_}_X:\A^X\to\P{X}$, which is clearly natural
  in the parameter set~$X$.
  From Prop.~\ref{p:CharacSeparA}, we have
  $$a\ent_{S[X]}b\quad\text{iff}\quad
  \bigmeet_{x\in X}(a_x\to b_x)~\in~S
  \quad\text{iff}\quad\rho_X(a)\le\rho_X(b)
  \eqno(\text{for all}~a,b\in\A^X)$$
  hence $\rho_X:\A^X\to\P{X}$ induces an embedding of
  posets $\tilde{\rho}_X:\P'{X}\to\P{X}$ through the quotient
  $\P'{X}:=\A^X/S[X]$.
  Moreover, the map $\tilde{\rho}_X:\P'{X}\to\P{X}$ is surjective
  (since $\rho_X$ is), therefore it is an isomorphism
  from the tripos~$\P'$ to the tripos~$\P$.
\end{proof}

The proof of Theorem~\ref{th:Thm} p.~\pageref{th:Thm} is now complete.

\subsection{The case of classical triposes}
In~\cite{Miq20}, we showed that each classical implicative tripos
(that is: a tripos induced by a classical implicative algebra~$\A$)
is isomorphic to a Krivine tripos (that is: a tripos induced by an
abstract Krivine structure).
Combining this result with Theorem~\ref{th:Thm}, we deduce that
classical realizability provides a complete description of all
classical triposes $\P:\Set^{\op}\to\BA$ (writing $\BA\subset\HA$ the
full subcategory of Boolean algebras):
\begin{corollary}
  Each classical tripos $\P:\Set^{\op}\to\BA$ is isomorphic to a
  Krivine tripos.
\end{corollary}

\subsection*{Acknowledgements}
We would like to thank \'Etienne Miquey and Luc Pellissier for their
useful remarks and corrections in an earlier version of this draft.

\bibliographystyle{plain}
\bibliography{paper}

\begin{thebibliography}{1}

\bibitem{Eng81}
E.~Engeler.
\newblock Algebras and combinators.
\newblock {\em Algebra Universalis}, 13(1):389--392, 1981.

\bibitem{HJP80}
J.~M.~E. Hyland, P.~T. Johnstone, and A.~M. Pitts.
\newblock Tripos theory.
\newblock In {\em Math. Proc. Cambridge Philos. Soc.}, volume~88, pages
  205--232, 1980.

\bibitem{Miq20}
A.~Miquel.
\newblock Implicative algebras: a new foundation for realizability and forcing.
\newblock {\em Math. Struct. Comput. Sci.}, 30(5):458--510, 2020.

\bibitem{Pit81}
A.~M. Pitts.
\newblock {\em The theory of triposes}.
\newblock PhD thesis, University of Cambridge, 1981.

\bibitem{Pit02}
A.~M. Pitts.
\newblock Tripos theory in retrospect.
\newblock {\em Math. Struct. Comput. Sci.}, 12(3):265--279, 2002.

\bibitem{Str13}
T.~Streicher.
\newblock Krivine's classical realisability from a categorical perspective.
\newblock {\em Math. Struct. Comput. Sci.}, 23(6):1234--1256, 2013.

\end{thebibliography}

\end{document}